\documentclass[11pt, reqno]{amsart}
\usepackage{amsmath,amsthm,amssymb,mathrsfs}
\usepackage[abbrev,non-sorted-cites]{amsrefs}
\usepackage{color,graphicx}
\usepackage{verbatim}
\usepackage{datetime}
\usepackage{hyperref}

\theoremstyle{plain}
  \newtheorem{thm}{Theorem}[section]
  \newtheorem{lem}[thm]{Lemma}
  \newtheorem{prop}[thm]{Proposition}
  \newtheorem{cor}[thm]{Corollary} 
	\newtheorem*{thm*a}{Theorem A}
	\newtheorem*{thm*b}{Theorem B}
\theoremstyle{definition}
  \newtheorem{defn}[thm]{Definition}
	
  \newtheorem{rmk}[thm]{Remark}
  \newtheorem*{ack*}{Acknowledgement}
  \newtheorem*{ques*}{Question}
\theoremstyle{plain}

\numberwithin{equation}{section}

\newcommand\ip[2]{\langle{#1},{#2}\rangle}

\newcommand\pl{\partial}

\newcommand\oh{\frac{1}{2}}
\newcommand\dd{{\mathrm d}}

\newcommand\sm{\sigma}
\newcommand\dt{\delta}
\newcommand\ep{\epsilon}
\newcommand\vep{\varepsilon}
\newcommand\vph{\varphi}

\newcommand\ta{\theta}

\newcommand\ld{\lambda}

\newcommand\Om{\Omega}

\newcommand\Gm{\Gamma}
\newcommand\Ld{\Lambda}

\newcommand\Dt{\Delta}

\newcommand\CP{\mathcal{P}}

\newcommand\BR{\mathbb{R}}
\newcommand\BN{\mathbb{N}}

\newcommand\td{\tilde}

\newcommand\bu{\mathbf{u}}
\newcommand\bv{\mathbf{v}}

\newcommand\bfI{\mathbf{I}}
\newcommand\ii{\sqrt{-1}}

\DeclareMathOperator{\tr}{tr}

\DeclareMathOperator{\supp}{supp}

\newcommand\ST{S^{[2]}}
\newcommand\heat{\frac{\pl}{\pl t} - \Dt}

\begin{document}

\title{Entire solutions of two-convex Lagrangian\\
mean curvature flows}

\author{Chung-Jun Tsai}
\address{Department of Mathematics, National Taiwan University, Taipei 10617, Taiwan}
\email{cjtsai@ntu.edu.tw}

\author{Mao-Pei Tsui}
\address{Department of Mathematics, National Taiwan University, Taipei 10617, Taiwan}
\email{maopei@math.ntu.edu.tw}

\author{Mu-Tao Wang}
\address{Department of Mathematics, Columbia University, New York, NY 10027, USA}
\email{mtwang@math.columbia.edu}

%\date{\usdate{\today}}

\thanks{C.-J.~Tsai is supported in part by the National Science and Technology Council grants 112-2636-M-002-003 and 112-2628-M-002-004-MY4.  M.-P.~Tsui is supported in part by the National Science and Technology Council grants 109-2115-M-002-006 and 112-2115-M-002-015-MY3. This material is based upon work supported by the National Science Foundation under Grant Numbers DMS-1810856 and DMS-2104212 (Mu-Tao Wang).  Part of this work was carried out when M.-T.~Wang was visiting the National Center of Theoretical Sciences.}

%%%%%%%%%%%%%%%%%%%%%%%%%%%%%%%%%%%%%%%%%%%%%%%%%%%%%%%%%%%%%%%%
\begin{abstract}
Given an entire $C^2$ function $u$ on $\mathbb{R}^n$, we consider the graph of $D u$ as a Lagrangian submanifold of $\mathbb{R}^{2n}$, and deform it by the mean curvature flow in $\mathbb{R}^{2n}$.  This leads to the special Lagrangian evolution equation, a fully nonlinear Hessian type PDE.  We prove long-time existence and convergence results under a 2-positivity assumption of $(I+(D^2 u)^2)^{-1}D^2 u$. Such results were previously known only under the stronger assumption of positivity of $D^2 u$.

\end{abstract}

\maketitle
%\tableofcontents

%%%%%%%%%%%%%%%%%%%%%%%%%%%%%%%%%%%%%%%%%%%%%%%%%%%%%%%%%%%%%%%%
\section{Introduction}

Fundamental to a fully nonlinear Hessian type 
partial differential equation is the convexity assumption, which guarantees ellipticity or parabolicity and facilitates essential regularity
estimates. Among these equations, the special 
Lagrangian equation is rather unique in that the variational structure always implies ellipticity or parabolicity without the need of any convexity assumption. Still most previous results were obtained under the convexity condition or its equivalence. In \cite{TTW22l}, a substantial improvement was achieved by removing the convexity assumption and replacing it by the two-convexity assumption. The main purpose of this paper is to extend results in \cite{TTW22l} from the compact setting to the non-compact global setting. Essential difficulties of this extension were dealt with by several new ingredients in this article. These include two new evolution equations (Proposition \ref{prop_starOm} and Proposition \ref{prop_detST}) which allow us
to localize estimates in \cite{TTW22l} in the global setting. 
There is also a new regularization procedure that works without the convexity assumption. In particular, new non-convex self-expanders of Lagrangian mean curvature flows were discovered.

Given an entire $C^2$ function $u_0$ on $\mathbb{R}^n$, we consider the graph of $D u_0$ as a Lagrangian submanifold of $\mathbb{R}^{2n}$ and deform it by the mean curvature flow, which corresponds to the following initial value problem for the potential function $u$:
\begin{align} \begin{split}
    \frac{\pl u}{\pl t} &= \frac{1}{\ii} \log \frac{\det(\bfI + \sqrt{-1}D^2u)}{\sqrt{\det(\bfI+(D^2u)^2)}} \\
    \text{with } & u(x,0) = u_0(x) ~.
\end{split} \label{LMCF0} \end{align}

As remarked above, this is a fully nonlinear Hessian type parabolic equation. 
We assume that $u_0$ is {\it 2-convex}, which is a 2-positivity assumption in terms of $(I+(D^2u_0)^2)^{-1} D^2 u_0$(see Definition \ref{def_2convex} for the definition in terms of the eigenvalues of $D^2u_0$). This is a natural quantity as it corresponds to the Hessian of $u$ as measured with respect to the induced metric on the corresponds Lagrangian submanifold.

The main result of this paper is the following long-time existence result.
\begin{thm}[Theorem \ref{thm_2convexini}] \label{thm_main}
Let $u_0\in C^2(\BR^n)$ be a $2$-convex function with $\sup_{x\in\BR^n}|D^2 u_0|^2\leq c$ for some $c>0$.   Then, \eqref{LMCF0} admits a unique solution $u(x,t)$ in the space $C^0(\BR^n\times[0,\infty))\cap C^\infty(\BR^n\times(0,\infty))$ such that
    \begin{itemize}
        \item for any $t>0$, $u$ is $2$-convex;
        \item there exists $c_\ell = c_\ell(c) >0$ for any $\ell\geq 2$ such that $\sup_{x\in\BR^n}|D^\ell u|^2 \leq c_\ell t^{2-\ell}$ for any $t>0$.
    \end{itemize}
\end{thm}

Depending on the asymptotic behavior of the initial data, we prove the following convergence theorems.

\begin{thm}[Theorem \ref{thm_bounded_gradient}]
    Let $u_0\in C^2(\BR^n)$ be a $2$-convex function with $\sup_{x\in\BR^n}|D^2 u_0|^2\leq c$ for some $c>0$.   Denote by $u(x,t)$ the solution to \eqref{LMCF0} with $u(x,0) = u_0(x)$ given by Theorem \ref{thm_main}.  Suppose that there exists a constant $c_1>0$ so that $|Du_0|^2\leq c_1$ on $\BR^n$.  Then, there exists an $\vec{a}\in\BR^n$ such that $Du(x,t)$ converges to the constant map from $\BR^n$ to $\vec{a}$ in $C^\infty_{\text{loc}}(\BR^n,\BR^n)$.  
\end{thm}

In other words, $L_u = \{(x,Du(x,t)):x\in\BR^n\}$ converges locally smoothly to $\BR^n\times\{\vec{a}\}$ as $t\to\infty$.

\begin{thm}[Theorem \ref{thm_conv_expander}]
    Let $u_0\in C^2(\BR^n)$ be a $2$-convex function with $\sup_{x\in\BR^n}|D^2 u_0|^2\leq c$ for some $c>0$, and
    \begin{align*}
        \lim_{\mu\to\infty} \frac{u_0(\mu x)}{\mu^2} &= U_0(x) ~,
    \end{align*}
    for some $U_0(x)$.  Let $u(x,t)$ be the solution to \eqref{LMCF0} given by Theorem \ref{thm_main}.  Then, $u(\mu x,\mu^2 t)/\mu^2$ converges to a smooth self-expanding solution $U(x,t)$ to \eqref{LMCF0} in $C^\infty_{\text{loc}}(\BR^n\times(0,\infty))$ as $\mu\to\infty$.  As $t\to 0$, $U(x,t)$ converges to $U_0(x)$ in $C^0_{\text{loc}}(\BR^n)$.
\end{thm}

The paper is organized as follows. In Section 2, we consider the geometry of a Lagrangian submanifold in terms of its potential function. In Section 3, we review some known results about Lagrangian mean curvature flows that are needed in the article. In Section 4, we derive two important evolution equations that play critical roles in the proof of long-time existence.  The long-time existence results are established in Section 5 and the convergence results in Section 6.

\begin{ack*}
     We thank anonymous referee for helpful comments, suggestions, and corrections.
\end{ack*}

\section{The Lagrangian Geometry in Potential}

Endow $\BR^{2n} = \BR^n\oplus\BR^n$ with the standard metric, symplectic form, and complex structure.
Given a function $u:\BR^n\to\BR$, the graph of its gradient is a Lagrangian submanifold in $\BR^{2n}$.  Denote it by $L_u = \{(x,Du(x)):x\in\BR^n\}$.

At any $x\in\BR^n$, one may find a orthonormal basis to diagonalize the Hessian of $u$, $D^2 u$.  Specifically, $D^2 u = \ld_i\dt_{ij}$ with respect to an orthonormal basis $\{a_i\}_{i=1,\ldots,n}$ for $\BR^n$.  It follows that the tangent space of $\Gm(Du)$ has the orthonormal basis
\begin{align}
    e_i &= \frac{1}{\sqrt{1+\ld_i^2}} \left( a_i + \ld_i J(a_i) \right)
\label{basis0} \end{align}
for $i=1,\ldots,n$.  In terms of the parametrization $x\in\BR^n\to(x,Du(x))\in\BR^n\oplus\BR^n$ (the so-called non-parametric form in the minimal graph theory), the induced metric has metric coefficients
\begin{align}
    g_{ij} &= (1+\ld_i^2)\dt_{ij} ~.
\label{metric0} \end{align}

We shall study two parallel tensors on $\BR^{2n}$.  The first tensor is the volume form of $\BR^n\oplus\{0\}\subset\BR^{2n}$.  It is an $n$-form on $\BR^{2n}$, and denote it by $\Om$.  The restriction of $\Om$ on $L_u$ is equivalent to the scalar-valued function $*\Om$, where $*$ is the Hodge star of the induced metric on $L_u$.  By using the frame \eqref{basis0},
\begin{align}
    *\Om &= \frac{1}{\sqrt{\prod_{i=1}^n(1+\ld_i^2)}} ~.
\label{starOm} \end{align}
It is clear that $*\Om$ takes value in $(0,1]$.

The second one is a $(0,2)$-tensor:
\begin{align}
    S(X,Y) &= \ip{J \pi_1(X)}{\pi_2(Y)} ~.
\label{Stensor} \end{align}
With respect to the frame \eqref{basis0}, the restriction of $S$ on $L_u$ is
\begin{align*}
    S_{ij} &= \frac{\ld_i}{1+\ld_i^2}\dt_{ij} ~.
\end{align*}
In particular, it is positive definite if and only if $u$ is convex.
The main interest of this paper is the $2$-positivity case.  Namely,
\begin{align}
    S_{ii}+S_{jj} &= \frac{(\ld_i+\ld_j)(1+\ld_i\ld_j)}{(1+\ld_i^2)(1+\ld_j^2)} > 0 ~.
\label{2convex0} \end{align}
for any $i\neq j$.  Note that $(\ld_i+\ld_j)(1+\ld_i\ld_j)>0$ does not correspond to a connected region in the $\ld_i\ld_j$-plane.

\begin{defn} \label{def_2convex}
    A $C^2$-function $u:\BR^n\to\BR$ is said to be \emph{$2$-convex} if the eigenvalues of its Hessian satisfy everywhere
    \begin{align*}
        \ld_i+\ld_j \geq 0 \quad\text{and}\quad 1+\ld_i\ld_j \geq 0
    \end{align*}
    for any $i\neq j$.  It is said to be \emph{strictly $2$-convex} if both inequalities are strict.
\end{defn}

As in \cite{CNS85}, we introduce a symmetric endomorphism on $\Ld^2TL_u$ to study the $2$-convexity of $u$.  It is denoted by $\ST$, and is given by
\begin{align*}
    \ST_{(ij)(k\ell)} &= S_{ik}\dt_{j\ell} + S_{j\ell}\dt_{ik} - S_{i\ell}\dt_{jk} - S_{jk}\dt_{i\ell}
\end{align*}
with respect to an orthonormal frame of $L_u$.  The $2$-convexity \eqref{2convex0} condition corresponds to the positivity of $\ST$, and thus it is useful to study the scalar valued function $\det\ST$.  In terms of \eqref{basis0},
\begin{align}
    \det\ST &= \prod_{i<j}(S_{ii} + S_{jj}) = \prod_{i<j}\frac{(\ld_i+\ld_j)(1+\ld_i\ld_j)}{(1+\ld_i^2)(1+\ld_j^2)} ~.
\label{detST} \end{align}
It is not hard to see that if $u$ is strictly $2$-convex, $\det\ST$ takes value within $(0,1]$.

When $\vep_1\leq*\Om\leq 1$ and $\vep_2\leq\det\ST\leq 1$, one can deduce that
\begin{align}
    \sum_i\ld_i^2 &\leq \vep_1^{-2} - 1 ~, \label{bound_slope} \\
    1+\ld_i\ld_j &\geq \frac{\vep_2}{\sqrt{2(\vep_1^{-2}-1)}} ~, \label{bound_2con1}\\
    \ld_i+\ld_j &\geq \frac{2\vep_2}{\vep_1^{-2}+1} \label{bound_2con2}
\end{align}
for any $i\neq j$ (under the strict $2$-convexity assumption); see \cite{TTW22l}*{section 2.3}.

%%%%%%%%%%%%%%%%%%%%%%%%%%%%%%%%%%%%%%%%%%%%%%%%%%%%%%%%%%%%%%%%

\section{The Lagrangian Mean Curvature Flow}

Given any $u_0:\BR^n\to\BR$, it is known that the Lagrangian mean curvature flow (up to a tangential diffeomorphism) equation on $L_{u_0}$ reduces to \eqref{LMCF0}  for the potential function $u$.

The uniqueness of the solution to \eqref{LMCF0} was established by Chen and Pang in \cite{CP09}.
 
\begin{rmk} \label{rmk_dangle}
    Consider the function
    \begin{align*}
        f(B) &= \frac{1}{\ii} \log \frac{\det(\bfI + \sqrt{-1}B)}{\sqrt{\det(\bfI+B^2)}}
    \end{align*}
    on the space of symmetric matrices.  A direct computation shows that its derivative is
    \begin{align*}
        \dd f(B) = \tr ((g_B)^{-1}\dd B)
    \end{align*}
    where $g_B = \bfI + B^2$.
\end{rmk}

%%%%%%%%%%%%%%%%%%%%%%%%%%%%%%%%

We recall some preliminary results about short-time existence and finite time singularity established in \cites{CCH12,CCY13}.

\begin{prop}[\cite{CCY13}*{Proposition 2.1}] \label{prop_short_time}
    Suppose that $u_0:\BR^n\to\BR$ is a smooth function with $\sup|D^\ell u_0|<\infty$ for any $\ell\geq 2$.  Then, \eqref{LMCF0} admits a smooth solution $u(x,t):\BR^n\times[0,T)\to\BR$ for some $T>0$.  Moreover, $\sup\{|D^\ell u(x,t)|:x\in\BR^n\}<\infty$ for any $\ell\geq 2$ and $t\in[0,T)$.
\end{prop}

\begin{lem}[\cite{CCH12}*{Lemma 4.2}] \label{lem_finite_sing}
    Let $u$ be a smooth solution to \eqref{LMCF0} on $\BR^n\times[0,T)$ for some $T>0$.  Suppose that
    \begin{itemize}
        \item $\sup\{|D^\ell u(x,t)|:x\in\BR^n\}<\infty$ for any $\ell\geq 2$ and $t\in[0,T)$;
        \item $D^2u$ and $D^3 u$ are uniformly bounded on $\BR^n\times[0,T)$.  Namely, there exist $c_2, c_3>0$ such that $|D^2u(x,t)|\leq c_2$ and $|D^3u(x,t)|\leq c_3$ on $\BR^n\times[0,T)$.
    \end{itemize}
    Then, there exists $c_\ell = c_\ell(c_2,c_3) > 0$ for any $\ell\geq 4$ such that $|D^\ell u(x,t)|\leq c_\ell$ on $\BR^n\times[0,T)$.
\end{lem}

\begin{cor} \label{cor_sing}
    Suppose that $u_0:\BR^n\to\BR$ is a smooth function with $\sup|D^\ell u_0|<\infty$ for any $\ell\geq 2$.  Let $u(x,t):\BR^n\times[0,T)\to\BR$ be the solution to \eqref{LMCF0} given by Proposition \ref{prop_short_time}, where $T$ is the maximal existence time.  If $T<\infty$ and $|D^2u(x,t)|\leq c_2$ for some $c_2>0$ on $\BR^n\times[0,T)$, then
    \begin{align*}
        \lim_{t\to T} \sup\{|D^3u(x,t')|: x\in\BR^n ,~ t'\leq t\} &=\infty ~.
    \end{align*}
\end{cor}

\begin{proof}
Suppose that the above limit is finite.  Due to Lemma \ref{lem_finite_sing}, all the higher order ($\ell\geq 2$) derivatives of $u$ are uniformly bounded on $\BR^n\times[0,T)$.  In particular, Proposition \ref{prop_short_time} applies to the family of initial data $\td{u}_k(x) = u(x,(1-2^{-k})T)$, and there is a $T'>0$ such that the solution starting from $\td{u}_k$ exists on the time interval $[0,T')$.  It follows that the solution from $u_0$ can be extended over time $T$, which contradicts to $T<\infty$.
\end{proof}

%%%%%%%%%%%%%%%%%%%%%%%%%%%%%%%%

We will also need the Liouville theorem and a priori estimate established by Nguyen and Yuan in \cite{NY11}.  Denote $\BR^n\times(-\infty,0]$ by $Q_\infty$.  For any $r>0$, denote $B_r(0)\times[-r^2,0]\subset Q_\infty$ by $Q_r$.

\begin{prop}[\cite{NY11}*{Proposition\footnote{Specifically, this is the first step in that proposition, which uses on the Krylov--Safonov H\"older inequality.} 2.1}] \label{prop_Liouville}
    Let $u$ be a smooth solution to \eqref{LMCF0} in $Q_\infty$.  Suppose that $D^2u$ is uniformly bounded over $Q_\infty$.  Then, $u$ is stationary.  In other words, the right hand side of \eqref{LMCF0} vanishes, and $L_u$ is a static special Lagrangian submanifold.
\end{prop}

\begin{thm}[\cite{NY11}*{Theorem 1.1}] \label{thm_apriori}
    There exists a constant $c>0$ with the following significance.  Let $u$ be a smooth solution to \eqref{LMCF0} in $Q_1$.  Suppose that $1+\ld_i\ld_j\geq0$ for any $i\neq j$ everywhere in $Q_1$.  Then,
    \begin{align}
        [\pl_t u]_{1,\oh;Q_\oh} + [D^2 u]_{1,\oh;Q_\oh} &\leq c\,||D^2u||_{L^\infty(Q_1)} ~.
    \end{align}
\end{thm}

The notation on the left hand side is the semi-norm:
\begin{align*}
    [f]_{1,\oh;Q_r} &= \sup\left\{ \frac{f(x',t') - f(x,t)}{\max\{|x'-x|,|t'-t|^\oh\}} : (x',t'),(x,t)\in Q_r \,\text{ and }\, (x',t')\neq(x,t)\right\} ~.
\end{align*}
In particular, Theorem \ref{thm_apriori} implies that $||D^3 u||_{L^\infty(Q_\oh)} \leq c\,||D^2u||_{L^\infty(Q_1)}$.

%%%%%%%%%%%%%%%%%%%%%%%%%%%%%%%%%%%%%%%%%%%%%%%%%%%%%%%%%%%%%%%%

\section{Two Evolution Equations}

We derive the evolution equations for $\log(*\Om)$ and $\log\det\ST$, defined in \eqref{starOm} and \eqref{detST}, respectively.  They will be considered in the parametric form: the parametrization given by $\frac{\pl}{\pl t}F = H$.  That is to say, suppose that $u:\BR^n\times[0,T)$ is a solution to \eqref{LMCF0}, then $F$ is $(x,Du(x,t))$ composing with a time-dependent diffeomorphism.

\begin{prop} \label{prop_starOm}
    Let $u:\BR^n\times[0,T)$ be a solution to \eqref{LMCF0}.  Suppose that $u$ is $2$-convex for all $t\in[0,T)$.  Then. $\log(*\Om)$ satisfies
    \begin{align}
        (\heat)\log(*\Om) &\geq \frac{1}{n}\left|\nabla\log(*\Om)\right|^2
    \end{align}
    in the parametric form of the mean curvature flow.
\end{prop}

\begin{proof}
By \cite{TW02}*{2nd equation on P.532},
\begin{align*}
    \nabla_k(*\Om) &= -(*\Om)(\sum_{i}\ld_i h_{iik})
\end{align*}
with respect to the orthonormal frame \eqref{basis0}.
By the Cauchy--Schwarz inequality,
\begin{align*}
    |\nabla\log(*\Om)|^2 &= \sum_{k} \left|\sum_{i}\ld_ih_{kii}\right|^2 \\
    &\leq n\sum_{i,k}\ld_i^2h_{kii}^2 = n\left[ \sum_{i}\ld_i^2h_{iii}^2 + \sum_{i\neq j}\ld_i^2h_{iij}^2 \right] ~.
\end{align*}
According to \cite{TTW22l}*{Proposition 2.2},
\begin{align*}
    &\quad (\heat)\log(*\Om) \\
    &\geq \sum_{i}(1+\ld_i^2)h_{iii}^2 + \sum_{i\neq j}(3 + \ld_i^2 + 2\ld_i\ld_j)h_{iij}^2 + \sum_{i<j<k}(6+2\ld_i\ld_j+2\ld_j\ld_k+2\ld_k\ld_i)h_{ijk}^2 \\
    &\geq \sum_{i}h_{iii}^2 + \sum_{i\neq j}(3 + 2\ld_i\ld_j)h_{iij}^2 + \sum_{i<j<k}(6+2\ld_i\ld_j+2\ld_j\ld_k+2\ld_k\ld_i)h_{ijk}^2 + \frac{1}{n}\left|\nabla\log(*\Om)\right|^2 ~.
\end{align*}
Since $1+\ld_i\ld_j\geq0$, it finishes the proof of this proposition.
\end{proof}

\begin{prop} \label{prop_detST}
    Let $u:\BR^n\times[0,T)$ be a solution to \eqref{LMCF0}.  Suppose that $u$ is strictly $2$-convex for all $t\in[0,T)$.  Then. $\log\det\ST$ satisfies
    \begin{align}
        (\heat)\log\det\ST &\geq 2|A|^2 + \frac{1}{n(n-1)}\left|\nabla\log\det\ST\right|^2
    \end{align}
    in the parametric form of the mean curvature flow.
\end{prop}

\begin{proof}
According to \cite{TTW22l}*{Proposition 2.1},
\begin{align*}
    &\quad (\heat)\log\det\ST \\
    &\geq \sum_{k}\sum_{i<j} \left[ 4h_{kij}^2 + \frac{(1+\ld_i^2)(1+\ld_j^2)}{(\ld_i+\ld_j)^2}(h_{kii}+h_{kjj})^2 + \frac{(1+\ld_i^2)(1+\ld_j^2)}{(1+\ld_i\ld_j)^2}(h_{kii}-h_{kjj})^2 \right] ~.
\end{align*}
Since
\begin{align*}
    (1+\ld_i^2)(1+\ld_j^2) &= (\ld_i+\ld_j)^2 + (1-\ld_i\ld_j)^2 \\
    &= (1+\ld_i\ld_j)^2 + (\ld_i-\ld_j)^2 ~,
\end{align*}
one finds that
\begin{align}
    &\quad (\heat)\log\det\ST \notag \\
    &\geq \sum_{k}\sum_{i<j} \left[ 4h_{kij}^2 + (h_{kii}+h_{kjj})^2 + (h_{kii} - h_{kjj})^2 \right. \notag \\
    &\qquad\qquad\quad \left. + \frac{(1-\ld_i\ld_j)^2}{(\ld_i+\ld_j)^2}(h_{kii}+h_{kjj})^2 + \frac{(\ld_i-\ld_j)^2}{(1+\ld_i\ld_j)^2}(h_{kii}-h_{kjj})^2 \right] \notag \\
    &\geq 2|A|^2 + \sum_{k}\sum_{i<j} \left[ \frac{(1-\ld_i\ld_j)^2}{(\ld_i+\ld_j)^2}(h_{kii}+h_{kjj})^2 + \frac{(\ld_i-\ld_j)^2}{(1+\ld_i\ld_j)^2}(h_{kii}-h_{kjj})^2 \right] ~. \label{eqn_logdetST}
\end{align}

With respect to the orthonormal frame \eqref{basis0},
\begin{align*}
    \nabla_k(S_{ii}+S_{jj}) &= \left( \frac{1-\ld_i^2}{1+\ld_i^2}h_{kii} + \frac{1-\ld_j^2}{1+\ld_j^2}h_{kjj} \right) ~.
\end{align*}
By the Cauchy--Schwarz inequality,
\begin{align}
    |\nabla\log\det\ST|^2 &= \sum_{k} \left| \sum_{i<j} (S_{ii}+S_{jj})^{-1}\nabla_k(S_{ii}+S_{jj}) \right|^2 \notag \\
    &\leq \frac{n(n-1)}{2} \sum_{k}\sum_{i<j}(S_{ii}+S_{jj})^{-2}|\nabla_k(S_{ii}+S_{jj})|^2 ~. \label{est_STgrad1}
\end{align}

Rewrite $\nabla_k(S_{ii}+S_{jj})$ as follows:
\begin{align*}
    &\quad \frac{1-\ld_i^2}{1+\ld_i^2}h_{kii} + \frac{1-\ld_j^2}{1+\ld_j^2}h_{kjj} \\
    &= \oh\left(\frac{1-\ld_i^2}{1+\ld_i^2}+\frac{1-\ld_j^2}{1+\ld_j^2}\right)(h_{kii}+h_{kjj}) + \oh\left(\frac{1-\ld_i^2}{1+\ld_i^2}-\frac{1-\ld_j^2}{1+\ld_j^2}\right)(h_{kii}-h_{kjj}) \\
    &= \frac{1-\ld_i^2\ld_j^2}{(1+\ld_i^2)(1+\ld_j^2)}(h_{kii}+h_{kjj}) - \frac{\ld_i^2-\ld_j^2}{(1+\ld_i^2)(1+\ld_j^2)}(h_{kii}-h_{kjj}) ~.
\end{align*}
Together with \eqref{2convex0}, it leads to
\begin{align}
    &\quad (S_{ii}+S_{jj})^{-2}|\nabla_k(S_{ii}+S_{jj})|^2 \notag \\
    &= \left[ \frac{1-\ld_i\ld_j}{\ld_i+\ld_j}(h_{kii}+h_{kjj}) - \frac{\ld_i-\ld_j}{1+\ld_i\ld_j}(h_{kii}-h_{kjj}) \right]^2 \notag \\
    &\leq 2\frac{(1-\ld_i\ld_j)^2}{(\ld_i+\ld_j)^2}(h_{kii}+h_{kjj})^2 + 2\frac{(\ld_i-\ld_j)^2}{(1+\ld_i\ld_j)^2}(h_{kii}-h_{kjj})^2 ~. \label{est_STgrad2}
\end{align}

Putting \eqref{est_STgrad1}, \eqref{est_STgrad2} and \eqref{eqn_logdetST} together finishes the proof of this proposition.
\end{proof}

%%%%%%%%%%%%%%%%%%%%%%%%%%%%%%%%%%%%%%%%%%%%%%%%%%%%%%%%%%%%%%%%
\section{Long-Time Existence}

%%%%%%%%%%%%%%%%%%%%%%%%%%%%%%%%
\subsection{A Maximum Principle}

We first establish a maximum principle, whose proof is based on the argument in \cite{EH91}*{Theorem 2.1}.  

\begin{lem} \label{lem_max}
    Let $F:L\times[0,T)\to\BR^N$ be a solution to the mean curvature flow\footnote{in the parametric form, $\frac{\pl}{\pl t}F = H$}.  Suppose that $v$ is a positive smooth function satisfying
    \begin{align}
        (\heat)\log v &\leq -q |\nabla \log v|^2
    \label{heat_formal} \end{align}
    for some constant $q>0$, and $\sup_{L\times\{t\}}v < \infty$ for every $t\in[0,T)$.  Then,
    \begin{align}
        \sup_{L\times\{t\}} v &\leq \sup_{L\times\{0\}} v
    \end{align}
    for every $t\in(0,T)$.
\end{lem}

\begin{proof}
Let $\dim L = n$.  It follows from $\frac{\pl}{\pl t} F = H$ that
\begin{align}
    (\heat)|F|^2 &= -2n ~.
\label{test0} \end{align}
Fix $R>0$, and consider the function $\vph = R^2 - |F|^2 - 2nt$.  It follows from \eqref{test0} that
\begin{align}
    \nabla\vph^2 &= 2\vph\,\nabla\vph ~, \label{test1} \\
    (\heat)\vph^2 &= -2\left|\nabla\vph\right|^2 \label{test2}
\end{align}

By \eqref{heat_formal}, the functions $w = v^{q} = \exp(q\log v) > 0$ obeys
\begin{align*}
    (\heat)w &= qw\left[(\heat)\log v\right] - w^{-1}|\nabla w|^2 \leq -2w^{-1}|\nabla w|^2 ~.
\end{align*}
Together with \eqref{test1} and \eqref{test2},
\begin{align}
    (\heat)(\vph^2w^2) &\leq -6|\nabla w|\vph^2 - 2\left|\nabla\vph\right|^2w^2 - 2\ip{\nabla \vph^2}{\nabla w^2} \notag \\
    &\quad + c \left( \vph^{-1}\ip{\nabla\vph}{\nabla(\vph^2w^2)} - 2w^2|\nabla\vph|^2 - \vph\ip{\nabla\vph}{\nabla w^2} \right) \notag \\
    &= c\,\vph^{-1}\ip{\nabla\vph}{\nabla(\vph^2w^2)} \notag \\
    &\quad - 6|\nabla w|\vph^2 - 2(1+c)\left|\nabla\vph\right|^2w^2 - (8+2c)\ip{\vph\nabla w}{w\nabla\vph} \label{rev01}
\end{align}
for any $c\in\BR$.

By taking $c=2$, \eqref{rev01} becomes $-6\,|\vph\nabla w + w\nabla\vph|^2$, and hence
\begin{align}
    (\heat)(\vph^2w^2) &\leq 2\vph^{-1}\ip{\nabla\vph}{\nabla(\vph^2w^2)} ~.
\end{align}
By replacing $\vph$ with $\vph_+ = \max\{\vph,0\}$, the computation remains valid.  Due to the maximum principle,
\begin{align}
    \sup_{L\times\{t\}}(\vph_+\cdot w) &\leq \sup_{L\times\{0\}} (\vph_+\cdot w) ~.
\label{local_est} \end{align}
By letting $R\to\infty$, it implies that $\sup_{L\times\{t\}} w \leq \sup_{L\times\{0\}} w$.
\end{proof}

%%%%%%%%%%%%%%%%%%%%%%%%%%%%%%%%

\subsection{Smooth Initial Condition}

\begin{prop} \label{prop_longt_smooth}
    Let $u_0:\BR^n\to\BR$ be a smooth function with $\sup|D^\ell u_0|<\infty$ for any $\ell\geq 2$.  Suppose that
    \begin{align}
       u_0 \text{ is strictly $2$-convex,}\quad  *\Om \geq \vep_1 \quad\text{and}\quad \det\ST \geq \vep_2
    \label{bound1} \end{align}
    for some $\vep_1,\vep_2\in(0,1)$.  Then, \eqref{LMCF0} admits a unique smooth solution $u(x,t):\BR^n\times[0,\infty)\to\BR$ which has the following properties.
    \begin{enumerate}
        \item The solution $u(x,t)$ is strictly $2$-convex for all $t$.  Specifically, \eqref{bound1} is preserved along the flow.
        \item For any $\ell>2$, there exists a constant $c_\ell = c_\ell(\vep_1)>0$ such that $\sup_{x\in\BR^n}|D^\ell u|^2 \leq c_\ell t^{2-\ell}$ for any $t>0$.
    \end{enumerate}
\end{prop}

\begin{proof}
%%%%%%%%

{\it Step 1.~Preservation of \eqref{bound1}.}
Denote by $T$ the maximal existence time.  Since
\begin{align*}
    \CP &= \{t\in[0,T): \text{\eqref{bound1} holds true everywhere on } L_{u(\cdot,t)} \}
\end{align*}
is a closed subset of $[0,T)$, it remains to show that if $t_0\in\CP$, $[t_0,t_0+\dt)\subset\CP$ for some $\dt = \dt(t_0)>0$.

It suffices to do it for $t_0 = 0$.  Note that \eqref{bound_slope}, \eqref{bound_2con1} and \eqref{bound_2con2} hold true everywhere at $t=0$.  According to Proposition \ref{prop_short_time}, $|D^\ell u(x,t)|$ is bounded over $\BR^n\times[0,T/2]$ for any $\ell\geq2$.  By \eqref{LMCF0} and Remark \ref{rmk_dangle}, $D^2u(x,t)$ is a small\footnote{uniformly in space} perturbation of $D^2u(x,0)$ for $t<\!<1$.  Recall that the spectrum of a symmetric matrix is stable with respect to small perturbations; see for instance \cite{Tao12}*{section 1.3} and the references therein.  Together with \eqref{bound_slope}, \eqref{bound_2con1} and \eqref{bound_2con2} at $t = 0$, one infers that there exists a $\dt \in (0,T/2]$ such that $u$ remains $2$-convex for $t\in[0,\dt)$.

According to Proposition \ref{prop_detST},
\begin{align*}
     (\heat)\log(\det\ST)^{-1} &\leq -\frac{1}{n(n-1)}\left|\nabla\log(\det\ST)^{-1}\right|^2
\end{align*}
for $t\in[0,\dt)$.  By applying Lemma \ref{lem_max} for $v = (\det\ST)^{-1}$, we find that $\det\ST \geq \vep_2$ for $t\in[0,\dt)$.

Similarly, Proposition \ref{prop_starOm} says that
\begin{align*}
    (\heat)\log(*\Om)^{-1} &\leq \frac{1}{n}\left|\nabla\log(*\Om)^{-1}\right|^2
\end{align*}
for $t\in[0,\dt)$.  It follows from Lemma \ref{lem_max} for $v = (*\Om)^{-1}$ that $(*\Om)\geq\vep_1$ for $t\in[0,\dt)$.  Hence, $\CP = [0,T)$.

%%%%%%%%
{\it Step 2.~Long Time Existence.}
The next step is to show the maximal existence time $T$ is infinity.  Suppose not.  By Corollary \ref{cor_sing},
\begin{align*}
    A(t) = \sup\{ |D^3u(x,t')| : x\in\BR^n, t'\leq t \}
\end{align*}
is unbounded as $t\to T$.  It follows that there exists a sequence $(x_k,t_k)\in\BR^n\times[0,T)$ such that
\begin{itemize}
    \item $t_k \to T$ as $k\to\infty$;
    \item $A(t_k) \to \infty$ as $k\to\infty$;
    \item $|D^3u(x_k,t_k)|\geq A(t_k)/2$ for all $k$.
\end{itemize}
Denote $A(t_k)$ by $\rho_k$.  Let
\begin{align*}
    \td{u}_k(y,s) &= \rho_k^2 \left[ u(\frac{y}{\rho_k},t_k+\frac{s}{\rho_k^2}) - u(0,t_k) - Du(0,t_k)\cdot\frac{y}{\rho_k} \right]
\end{align*}
for $y\in\BR^n$ and $s\in[-\rho_k^2t_k,0]$.  Then one has $\td{u}_k(0,0) = 0 = D\td{u}_k(0,0)$, $|D^3\td{u}(0,0)| \geq 1/2$, and $|D^3\td{u}(y,t)|\leq 1$ on $\BR^n\times[-\rho_k^2t_k,0]$.  It is straightforward to see that $\td{u}_k$ solves \eqref{LMCF0} on $\BR^n\times[-\rho_k^2t_k,0]$.   Note that the eigenvalues of $D^2_y\td{u_k}$ is the same as those of $D^2_x u$.

By using Lemma \ref{lem_finite_sing} and the Arzel\`{a}--Ascoli theorem, $\td{u}_k$ admits a subsequence which converges to $\td{u}:\BR^n\times(-\infty,0]$ in $C^{\infty}_{\text{loc}}$.  Hence, $\td{u}$ is an ancient solution to \eqref{LMCF0} with $|D^3\td{u}(0,0)|\geq 1/2$, and the eigenvalues of $D^2\td{u}$ satisfies \eqref{bound_slope}, \eqref{bound_2con1} and \eqref{bound_2con2}.  Due to Proposition \ref{prop_Liouville}, $\td{u}$ is stationary.  Since $1+\ld_i\ld_j \geq 0$ and $\ld_i$'s are bounded, \cite{TW02}*{Theorem A} asserts that $L_{\td{u}}$ is an affine $n$-plane.  This contradicts to $|D^3\td{u}(0,0)|\geq 1/2$.  Thus, the maximal existence time $T$ cannot be finite.

%%%%%%%%
{\it Step 3.~Estimates.}  The argument for assertion (ii) follows from the a priori estimate \cite{NY11}*{Theorem 1.1} and the scaling argument; see \cite{CCY13}*{p.173}.  The $\ell=3$ case is included here for completeness.  Fix $(x,t)\in\BR^n\times(0,\infty)$.  Let
\begin{align*}
    \hat{u}(y,s) &= \frac{1}{t}u(x+t^\oh y,t(1+s)) ~.
\end{align*}
Then, $\hat{u}$ is a solution to \eqref{LMCF0} on $Q_1 = B_1(0)\times[-1,0]$.  Since the second order derivative remains unchanged in the rescaling, it follows from Theorem \ref{thm_apriori} that $||D^3\hat{u}||_{L^\infty}(Q_\oh) \leq c||D^2\hat{u}||_{L^\infty}(Q_1)$.  Due to \eqref{bound_slope}, $|D^2\hat{u}|$ is uniformly bounded.  Therefore,
\begin{align*}
    |D^3u(x',t')| &\leq c(\vep_1)\,t^{-\oh}
\end{align*}
for any $(x',t')$ with $|x'-x|\leq t^\oh/2$ and $3t/4 \leq t'\leq t$.  The estimate holds true for any $(x,t)$, and it finishes the proof for $\ell = 3$.  The estimate for $\ell>3$ follows a scaling argument; see the last part in the proof of \cite{CCH12}*{Lemma 5.2}.
\end{proof}

%%%%%%%%%%%%%%%%%%%%%%%%%%%%%%%%

\subsection{$C^2$ Initial Condition}

With Proposition \ref{prop_longt_smooth}, we can now prove the main theorem of this paper.  In \cite{CCH12}*{section 5 and 6} and \cite{CCY13}*{section 3}, the long-time existence was proved by constructing smooth approximations to the initial condition.  Unlike their situation, the convolution between the standard mollifier and a $2$-convex function needs not to be $2$-convex.  It requires some extra work to handle this issue.

The following lemma says that for a \emph{strict} $2$-convex region, one can find a positive cone such that the corresponding translation leaves the region invariant.  In fact, the slope of the cone is determined by the corner points of the boundary of the region.  See Figure \ref{pic1} for the $2$-convex region, and Figure \ref{pic2} for a strictly $2$-convex region.

\begin{lem} \label{lem_pos_cone}
    For any $\dt_1\in(0,1)$ and $\dt_2>0$, let $Q_{\dt_1,\dt_2} = \{(x,y)\in\BR^2: 1+xy\geq\dt_1 ,\, x+y\geq\dt_2\}$.  There exists a $\tau>1$ depending on $\dt_1,\dt_2$ such that    \begin{align*}
        Q_{\dt_1,\dt_2} + C_\tau \subset Q_{\dt_1,\dt_2}
    \end{align*}
    where $C_\tau = \{ (x,y): x\geq 0 ,\, x/\tau\leq y\leq \tau x \}$.  The above sum of two sets is the Minkowski sum.
\end{lem}

\begin{figure}
\centering
    \begin{minipage}[b]{0.4\textwidth}
    \includegraphics[width=\textwidth]{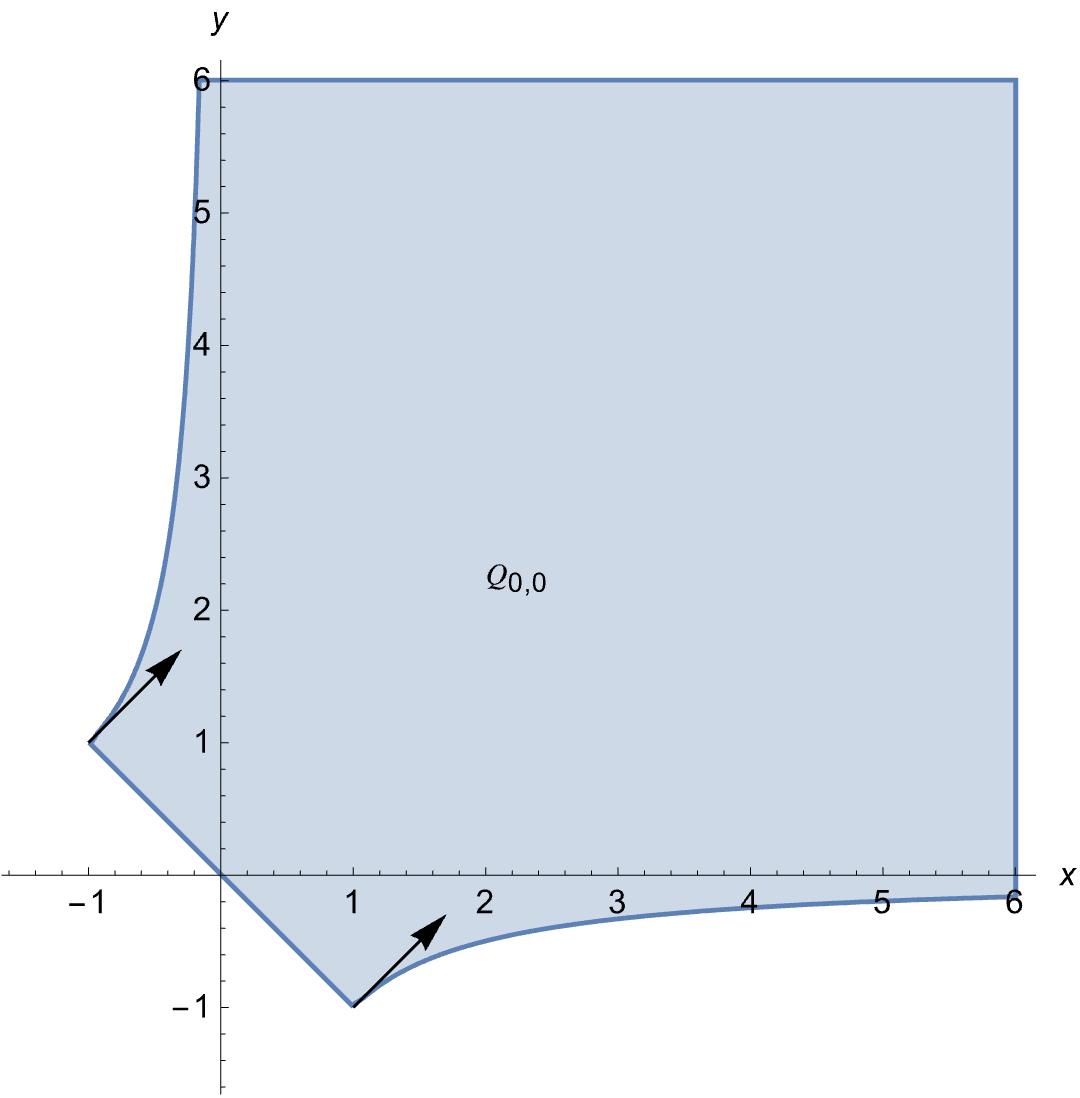}
    \caption{$Q_{\dt_1,\dt_2}$ for $\dt_1 = 0 = \dt_2$}
    \label{pic1}
    \end{minipage}
\hspace{1cm}
    \begin{minipage}[b]{0.4\textwidth}
    \includegraphics[width=\textwidth]{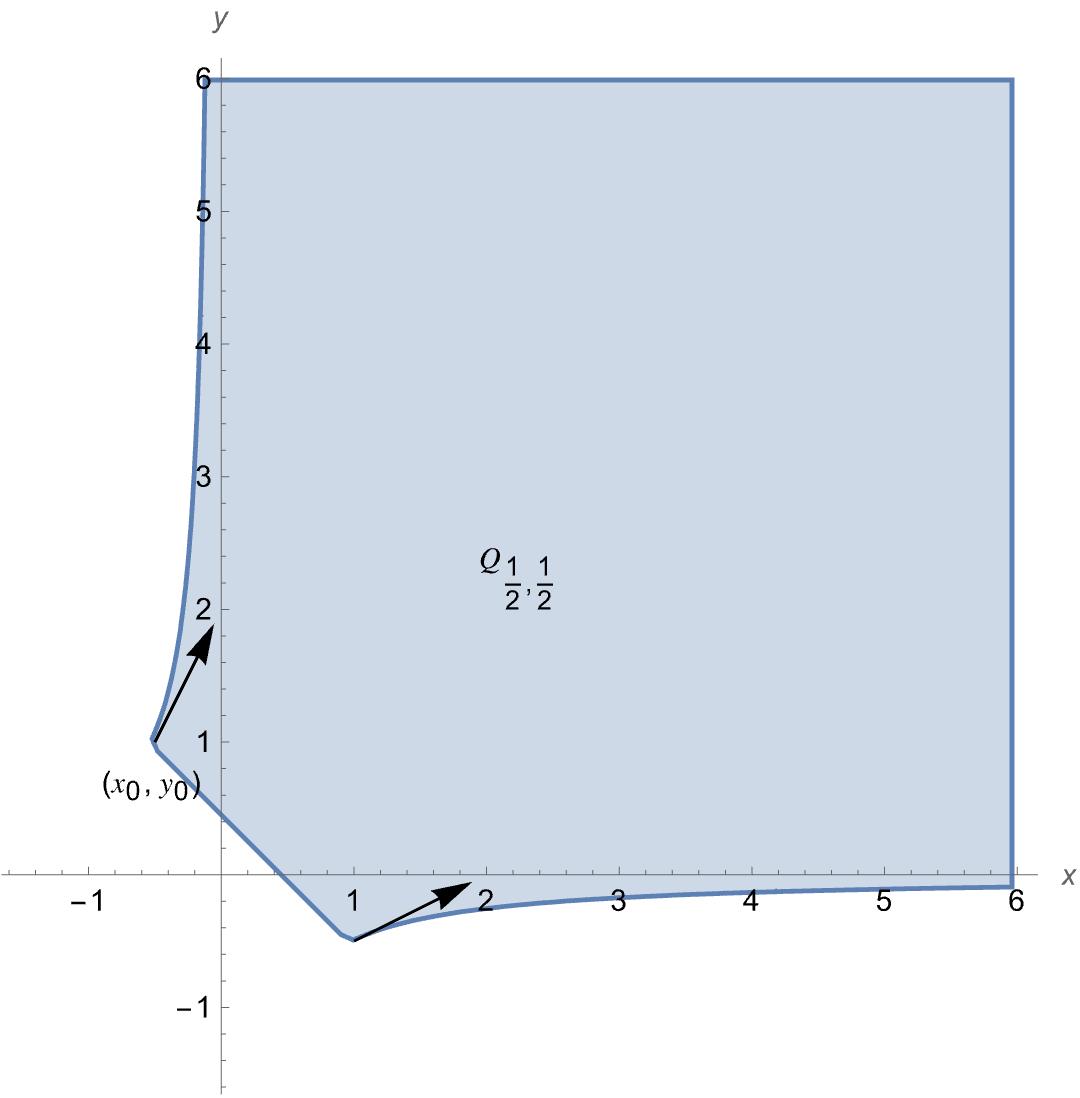}
    \caption{$Q_{\dt_1,\dt_2}$ for $\dt_1 = 1/2 = \dt_2$}
    \label{pic2}
    \end{minipage}
\end{figure}

%\parbox{12cm}{\center
%    \parbox{5.7cm}{\center{\includegraphics[width=0.25\textwidth]{region1.jpg}}}
%    \parbox{5.7cm}{\center{\includegraphics[width=0.25\textwidth]{region2.jpg}}}
%}

\begin{proof}
Since $\dt_1\in(0,1)$ and $\dt_2>0$, $1+xy=\dt_1$ and $x+y=\dt_2$ have exactly two intersection point; one in the second quadrant, and the other in the fourth quadrant.  Denote the intersection point in the second quadrant by $(x_0,y_0)$.  It must belong to the region $\Om = \{(x,y)\in\BR^2: -1<x<0 ,\, -x<y<-1/x \}$.  It is not hard to see that $(\dt_1,\dt_2)\in(0,1)\times(0,\infty) \mapsto (x_0,y_0)\in\Om$ is a one-to-one correspondence.  The inverse map is given by $\dt_1 = 1+x_0y_0$ and $\dt_2 = x_0+y_0$.

Let
\begin{align*}
    \tau &= -\frac{y_0}{x_0} ~.
\end{align*}
Pick any $(u,v)\in C_\tau\setminus\{(0,0)\}$.  It is clear that $(x+u)+(y+v) > x_0+y_0$ if $x+y\geq x_0+y_0$.  It remains to verify that $(x+u)(y+v) \geq x_0y_0$ if $xy\geq x_0y_0$ and $x+y\geq x_0+y_0$.  Note that the condition implies that $x\geq x_0$.  Assume that $x<0$, 
\begin{align*}
    (x+u)(y+v)- (x_0y_0) &> u \left(y + x\frac{v}{u}\right) \geq u\left(y + x\tau\right) \\
    &\geq u\left(y - x\frac{y_0}{x_0}\right) \\
    &\geq u\left(x_0+y_0-x - x\frac{y_0}{x_0}\right) = u(x_0+y_0)\frac{x-x_0}{-x_0} \geq0 ~.
\end{align*}
The argument for $y<0$ is similar.  If $x\geq0$ and $y\geq0$, it is obvious.
\end{proof}

Lemma \ref{lem_pos_cone} allows to construct some functions $F_k$ such that $u_0+F_k$ is still $2$-convex, and becomes \emph{convex} outside $B_k(0)$.

\begin{lem} \label{lem_booster}
    Given any $\tau>1$,  there exists a sequence of smooth functions $\{F_k(x)\}_{k\in\BN}$ with the following significance.
    \begin{enumerate}
        \item $D^2F_k = \bfI$ when $|x|\geq k$.
        \item $0< D^2F_k\leq \tau$ everywhere.
        \item At any $x\in\BR^n$, the ratio between any two eigenvalues of $D^2F_k(x)$ belongs to $[1/\tau,\tau]$.
        \item For any $R>0$, $F_k$ converges to $0$ uniformly over $\overline{B_R(0)}$ as $k\to\infty$, so do their derivatives.
    \end{enumerate}
\end{lem}

\begin{proof}
Let $r$ be the distance to the origin.  For $F = F(r)$, its Hessian has two eigenvalues, $F''$ and $r^{-1}F'$.  The first one is simple, and the second one has multiplicity $n-1$.  Denote $F'(r)$ by $f(r)$.
Let
\begin{align}
    u(r) &= f'(r) \frac{r}{f(r)} ~,
\label{ratio1} \end{align}
and then
\begin{align}
    f(r) &= f(r_0) \exp\left(\int_{r_0}^r \frac{u(\rho)}{\rho}\dd\rho \right) ~.
\label{ODE1} \end{align}
Thus, one only needs to specify $u(r)$, $r_0$, and $f(r_0)$ to construct $f(r)$.  By requiring $F(0)=0$, the function $F(r)$ is uniquely determined.

Fix a constant $\ta\in(0,1/10)$.  For any $k\in\BN$, choose a smooth function $u_k(r)$ for $r\geq0$ with
\begin{itemize}
    \item $u_k(r) = 1$ if $r\leq\ta$ or $r\geq k$;
    \item $u_k(r) = \tau$ if $2\ta\leq r\leq k/2$;
    \item $u'_k(r)\geq 0$ if $\ta\leq r\leq2\ta$;
    \item $u'_k(r)\leq 0$ if $k/2\leq r\leq k$.
\end{itemize}
The base point $r_0$ is set to be $k$, and the base value $f_k(k)$ is set to be $k$.  One can see Figure \ref{pic3} for an illustration of some $f_k(r)$'s (by using piecewise constant $u_k$).

It follows that $f_k(r) = r$ for $r\geq k$.  When $r\leq k$,
\begin{align*}
    f_k(r) &= k\exp\left(-\int_r^k \frac{u(\rho)}{\rho}\dd\rho\right)
    \leq k\exp\left(-\int_r^k \frac{1}{\rho}\dd\rho\right) = r ~.
\end{align*}
When $2\ta\leq r\leq k/2$,
\begin{align*}
    f_k(r) &= f_k(k/2)\exp\left(-\int_r^{k/2} \frac{u(\rho)}{\rho}\dd\rho\right) \leq \left(\frac{2}{k}\right)^{\tau-1} {r^\tau} ~.
\end{align*}
With a similar argument, $f_k(r) = c_k r$ when $r\leq \ta$, where
\begin{align*}
    c_k &= \frac{f_k(\ta)}{\ta} \leq \frac{f_k(2\ta)}{2\ta} \leq \left(\frac{4\ta}{k}\right)^{\tau-1} ~.
\end{align*}
It is not hard to see that the corresponding $F_k(r)$ satisfies the assertions of this lemma.
\end{proof}

\begin{figure}
\centering
    \includegraphics[width=0.4\textwidth]{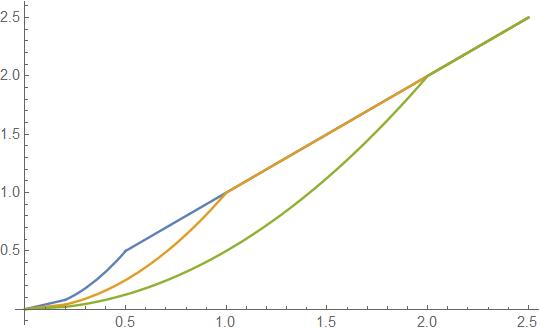}
    \caption{$F'_k(r)$}
    \label{pic3}
\end{figure}

Note that Lemma \ref{lem_booster} (ii) and (iii) implies that $( \ip{(D^2F_k)\bu}{\bu}, \ip{(D^2F_k)\bv}{\bv} )$ belongs to $C_\tau$ defined in Lemma \ref{lem_pos_cone}, for any unit vectors $\bu,\bv\in\BR^n$.

\begin{thm} \label{thm_longt_C2}
    Let $u_0\in C^2(\BR^n)$ satisfying \eqref{bound1} for some $\vep_1,\vep_2\in(0,1)$.  Then, \eqref{LMCF0} admits a unique solution $u(x,t)$ in the space $C^0(\BR^n\times[0,\infty))\cap C^\infty(\BR^n\times(0,\infty))$ such that
    \begin{itemize}
        \item \eqref{bound1} is preserved along the flow;
        \item there exists $c_\ell = c_\ell(\vep_1) >0$ for any $\ell>2$ such that $\sup_{x\in\BR^n}|D^\ell u|^2 \leq c_\ell t^{2-\ell}$ for any $t>0$.
    \end{itemize}
\end{thm}

\begin{proof}
Let
\begin{align*}
    \dt_1 = \min\left\{\frac{\vep_2}{\sqrt{2(\vep_1^{-2}-1)}},\vep_2\right\}  \quad\text{and}\quad  \dt_2 = \frac{2\vep_2}{\vep_1^{-2}+1} ~.
\end{align*}
By \eqref{bound_2con1} and \eqref{bound_2con2}, the eigenvalues of $D^2u_0$ satisfy $1+\ld_i\ld_j\geq\dt_1$ and $\ld_i+\ld_j\geq\dt_2$ for any $i\neq j$.  Let $\tau = \tau(\dt_1,\dt_2)$ be given by Lemma \ref{lem_pos_cone}.  Apply Lemma \ref{lem_booster} for this $\tau$ to get a sequence of functions, $\{F_k(x)\}_{k\in\BN}$.

It follows from Lemma \ref{lem_pos_cone}, Lemma \ref{lem_booster} and \eqref{bound_slope} that the eigenvalues of $D^2(u_0+F_k)$ satisfy
\begin{align} \begin{split}
    &1+\ld_i\ld_j\geq\dt_1 ~,~~ \ld_i+\ld_j\geq\dt_2 ~\text{ for any } i\neq j ~, \\
    &\text{and }~ \sum_i|\ld_i|^2\leq 2(\vep_1^{-2}-1+n\tau^2) ~.
\end{split} \label{bound2} \end{align}
The reason is that the spectrum of a symmetric matrix is stable with respect to small perturbations; see for instance \cite{Tao12}*{section 1.3} and the references therein.  A complete discussion on the perturbation theory of eigenvalues of matrices can be found in \cite{Kato76}*{ch.1 and 2}.  By \eqref{bound2}, there exist $\vep'_1,\vep'_2\in(0,1)$ depending\footnote{The dependence on the dimension $n$ is always omitted in this paper.} on $\vep_1,\vep_2$ such that
\begin{align}
    *\Om \geq \vep'_1 \quad\text{and}\quad \det\ST \geq \vep'_2
\label{bound2a} \end{align}
for $u_0 + F_k$.  Let $A$ be a symmetric matrix whose eigenvalues satisfy \eqref{bound2}.  It is not hard to see that there exists $\dt_3>0$ such that $\dt_3 \leq A+\bfI\leq 1/\dt_3$.  Note that $\{B:\text{ symmetric }n\times n\text{ matrix}: \dt_3\leq B\leq 1/\dt_3\}$ is a convex set, and there exists $\vep''_1, \vep''_2\in(0,1)$ such that
\begin{align}
    *\Om \geq \vep''_1 \quad\text{and}\quad \det\ST \geq \vep''_2
\label{bound2b} \end{align}
for any symmetric matrix\footnote{The functions $\log(*\Om)$ and $\log\det\ST$ are defined for a symmetric matrix by using \eqref{starOm} and \eqref{detST}, respectively.} $B$ with $\dt_3\leq B\leq 1/\dt_3$.

Let $\{\eta_\sm\}_{\sm>0}$ be the standard mollifiers; see for instance \cite{Evans98}*{Appendix C.4}.  They have the following properties: $\eta_\sm\geq0$, $\supp\eta_\sm\subset \overline{B_\sm(0)}$, $\int_{\BR^n}\eta_\sm = 1$ for all $\sm>0$, and $\eta_\sm(x)$ depends only on $|x|$.  For any $k\in\BN$ and $\sm>0$, let
\begin{align}
    \td{w}_0^{k,\sm}(x) &= \int_{\BR^n} \left(u_0(y)+F_k(y)\right)\,\eta_\sm(x-y)\,\dd y ~.
\label{approx_ini} \end{align}
We claim that there exists a sequence of positive numbers $\{\sm_k\}_{k\in\BN}$ with $\sm_k\to0$ as $k\to\infty$ such that $\td{w}_0^{k,\sm_k}$ obeys
\begin{align}
    *\Om \geq \min\{\oh\vep'_1,\vep''_1\} \quad\text{and}\quad \det\ST \geq \min\{\oh\vep'_2,\vep''_2\} ~.
\label{bound2c} \end{align}
Note that
\begin{align*}
    D_x^2\td{w}_0^{k,\sm}(x) &= \int_{\BR^n} D^2_y\left(u_0(y)+F_k(y)\right)\,\eta_\sm(x-y)\,\dd y ~.
\end{align*}
When $|x|\geq k+1$, it follows from Lemma \ref{lem_booster} (i) and \eqref{bound2b} that for $\td{w}_0^{k,\sm}$, $*\Om|_x < \vep''_1$ and $\det\ST|_x \geq \vep''_2$ for any $\sm_k<1$.  When $|x|\leq k+1$,  it follows from \eqref{bound2}, \eqref{bound2a} and the uniform continuity of $D^2(u_0+F_k)$ on $\overline{B_{k+2}(0)}$ that \eqref{bound2c} holds true for $\td{w}_0^{k,\sm_k}$ on $B_{k+1}(0)$, provided that $\sm_k$ is sufficiently small.  This finishes the proof of the claim.

Now, set $w^k_0$ to be $\td{w}_0^{k,\sm_k}$.  With \eqref{bound2c} and
\begin{align*}
    D_x^\ell w_0^k(x) &= \int_{\BR^n} D^2_y\left(u_0(y)+F_k(y)\right)\,D^{\ell-2}_x\eta_{\sm_k}(x-y)\,\dd y ~,
\end{align*}
for any $\ell\geq 2$, one can apply Proposition \ref{prop_longt_smooth} to the initial conditions $w_0^k$.  Denote the solution by $w^k$, and it obeys \eqref{bound2c} for all $t$.

It follows from \eqref{bound2c} and \eqref{bound_slope} that $D^2w^k$ is uniformly bounded on $\BR^n\times[0,\infty)$.  Since $w^k$ solves \eqref{LMCF0}, $\frac{\pl w^k}{\pl t}$ is also uniformly bounded over $\BR^n\times[0,\infty)$.  Fix $R>0$ and $T>0$.  Let $K_{R,T} = \overline{B_R(0)}\times[0,T]\subset\BR^n\times[0,\infty)$.  It follows from \eqref{approx_ini} and Lemma \ref{lem_booster} (iv) that $w^k_0$ is uniformly bounded over $\overline{B_R(0)}$.  With the uniform boundedness of the time derivative, $w^k$ is uniformly bounded over $K_{R,T}$.  Together with the uniform boundedness of $D^2w^k$, $Dw^k$ is uniformly bounded over $K_{R,T}$.  Hence, there exists a subsequence of $w^k$ which converges uniformly on any compact subset of $\BR^n\times[0,\infty)$.  By Lemma \ref{lem_booster} (iv) and $\sm_k\to 0$ as $k\to\infty$, $\lim_{k\to\infty}w^k_0(x) = u_0(x)$.

This together with assertion (ii) of Proposition \ref{prop_longt_smooth} implies that (a subsequence of) $w^k(x,t)$ converges smoothly on any compact subset of $\BR^n\times(0,\infty)$.  The limit $u$ is a solution to \eqref{LMCF0} with initial condition $u_0$.  The decay estimate on the derivatives of $u$ follows from that of $w^k$.  Since the constant $c_\ell$ depends on the $|D^2w^k|$, it is not hard to see from the last item of \eqref{bound2} that $c_\ell$ depends only on $\vep_1$.

By construction, $u$ satisfies \eqref{bound2c} for all $t>0$.  To see that $u$ still satisfies \eqref{bound1}, we will need the local estimate \eqref{local_est} in Lemma \ref{lem_max}.  Fix an ${\ep_0}>0$.  For any $(x_0,t_0)\in\BR^n\times(0,\infty)$, let $K_0$ be its compact neighborhood $\overline{B_1(x_0)}\times[t_0/2,2t_0]\subset\BR^n\times(0,\infty)$.  Since $Dw^k$ is uniformly bounded over $K_0$, there exists an $R>0$ such that
\begin{align*}
    \vph = R^2 - |x|^2 - |Dw^k|^2 - 2nt \geq \frac{R^2}{1+\ep_0}
\end{align*}
on $K_0$.  Note that this $\vph$ is equal to the one in the proof of Lemma \ref{lem_max}.  By Proposition \ref{prop_starOm} and Lemma \ref{lem_max}, \eqref{local_est} for $v= (*\Om)^{-1}$ gives that
\begin{align*}
    \max_{(x,t)\in K_0}\left[(*\Om)\text{ of }w^k\right]^{-\frac{1}{n}} &\leq (1+\ep_0) \max_{x\in\overline{B_R(0)}} \left[(*\Om)\text{ of }w_0^k\right]^{-\frac{1}{n}} ~.
\end{align*}
Due to the uniform continuity of $D^2u_0$ over $\overline{B_R(0)}$, Lemma \ref{lem_booster} (iv), and $\sm_k\to0$,
\begin{align*}
    \limsup_{k\to\infty}\max_{x\in\overline{B_R(0)}} \left[(*\Om)\text{ of }w_0^k\right]^{-1} \leq \frac{1}{\vep_1} ~.
\end{align*}
It follows that $(*\Om)^{-1}$ of $u$ over $K_0$ is no greater than $(1+\ep_0)^n/\vep_1$.  Since it is true for any $\ep_0>0$, $*\Om$ of $u$ is always greater than or equal to $\vep_1$.  The argument for $\det\ST$ is similar.
\end{proof}

Since a $2$-convex initial can be approximated by strictly $2$-convex functions, we can extend Theorem \ref{thm_longt_C2} to $2$-convex initial condition.

\begin{thm} \label{thm_2convexini}
    Let $u_0\in C^2(\BR^n)$ be a $2$-convex function with $\sup_{x\in\BR^n}|D^2 u_0|^2\leq c$ for some $c>0$.   Then, \eqref{LMCF0} admits a unique solution $u(x,t)$ in the space $C^0(\BR^n\times[0,\infty))\cap C^\infty(\BR^n\times(0,\infty))$ such that
    \begin{itemize}
        \item for any $t>0$, $u$ is $2$-convex and $*\Om\geq \vep_1$;
        \item there exists $c_\ell = c_\ell(\vep_1) >0$ for any $\ell>2$ such that $\sup_{x\in\BR^n}|D^\ell u|^2 \leq c_\ell t^{2-\ell}$ for any $t>0$.
    \end{itemize}
\end{thm}

\begin{proof}
Consider $w^k_0 = u_0 + e^{-k}|x|^2$ for $k\in\BN$.  Since $|D^2u_0|\leq C$, it is not hard to see that there exist $\vep'_1 = \vep'_1(\vep_1)>0$ and $\vep'_{2,k} = \vep'_{2,k}(\vep_1,k)>0$ such that $*\Om\geq \vep'_1$ and $\det\ST\geq\vep'_{2,k}$ for $w^k_0$.  Denote by $w^k$ the solution to \eqref{LMCF0} with initial condition $w^k_0$ given by Theorem \ref{thm_longt_C2}.  Note that the bound of the derivatives given by Theorem \ref{thm_longt_C2} only depends on $\vep_1$.  By the same argument as that in the proof of Theorem \ref{thm_longt_C2}, there exists a subsequence of $w^k$ which converges in $C^0_{\text{loc}}(\BR^n\times[0,\infty))$ and in $C^\infty(\BR^n\times(0,\infty))$.  The properties asserted by this theorem follows from $\liminf_{k\to\infty}(*\Om \text{ of }w^k_0)\geq \vep_1$ and the smooth convergence of $w^k$ in $C^\infty(\BR^n\times(0,\infty))$.
\end{proof}

%%%%%%%%%%%%%%%%%%%%%%%%%%%%%%%%%%%%%%%%%%%%%%%%%%%%%%%%%%%%%%%%
\section{Some Convergence Results}

%%%%%%%%%%%%%%%%%%%%%%%%%%%%%%%%
\subsection{Potentials of Bounded Gradient}

\begin{thm}\label{thm_bounded_gradient}
    Let $u_0\in C^2(\BR^n)$ be a $2$-convex function with $\sup_{x\in\BR^n}|D^2 u_0|^2\leq c$ for some $c>0$.  Denote by $u(x,t)$ the solution to \eqref{LMCF0} with $u(x,0) = u_0(x)$ given by Theorem \ref{thm_2convexini}. Then, there exists an $\vec{a}\in\BR^n$ such that $Du(x,t)$ converges to the constant map from $\BR^n$ to $\vec{a}$ in $C^\infty_{\text{loc}}(\BR^n,\BR^n)$.  In other words, $L_u = \{(x,Du(x,t)):x\in\BR^n\}$ converges locally smoothly to $\BR^n\times\{\vec{a}\}$ as $t\to\infty$.
\end{thm}

\begin{proof}
The key step is to show that the bounded gradient condition is preserved along the flow.  Denote $\pl u/\pl{x^k}$ by $u_k$, $\pl^2u/\pl x^i\pl x^j$ by $u_{ij}$, etc.  By \eqref{LMCF0} and Remark \ref{rmk_dangle},
\begin{align}
    \frac{\pl}{\pl t} u_k &= g^{ij}u_{ijk}
\label{eqn_1der} \end{align}
where $g^{ij}$ is the inverse of $g_{ij} = \dt_{ij} + \sum_k u_{ik}u_{jk}$.  It follows that
\begin{align*}
    \frac{\pl}{\pl t}\big[c-\sum_k (u_k)^2\big] &= -2\sum_k u_k g^{ij}u_{ijk} \\
    &= g^{ij}\pl_i\pl_j\big[c-\sum_k(u_k)^2\big] + 2\sum_k g^{ij}u_{ki}u_{kj} ~,
\end{align*}
and hence
\begin{align}
    \frac{\pl}{\pl t}\big[c-\sum_k (u_k)^2\big] - g^{ij}\pl_i\pl_j\big[c-\sum_k(u_k)^2\big] &\geq 0 ~.
\end{align}

Since $*\Om\geq\vep_1$ along the flow, $g^{ij}$ is uniformly bounded over $\BR^n\times[0,\infty)$.  According to Theorem \ref{thm_2convexini}, $t|D^3u|^2$ is uniformly bounded.  With \eqref{eqn_1der}, there exits for any $T>0$ a constant $C_T$ such that $|u_k|\leq C_T$ on $\BR^n\times[0,T]$.  Thus, there is a $C'_T>0$ such that $\big[c-\sum_k(u_k)^2\big] \geq -C'_T$ on $\BR^n\times[0,T]$.  By the maximum principle in \cite{Fr64}*{Theorem 9 on p.43}, $c-\sum_k(u_k)^2\geq 0$ on $\BR^n\times[0,T]$.

With $|D^\ell u|\leq c_\ell t^{2-\ell}$ for any $\ell>2$, $Du(x,t)$ converges in $C^\infty_{\text{loc}}(\BR^n,\BR^n)$ to some $p(x)$ as $t\to\infty$.  Since the convergence is $C^\infty_{\text{loc}}$, the image of $(x,p(x))$ is still a minimal Lagrangian submanifold.  Thus, $p(x) = D\bar{u}(x)$ for some $\bar{u}(x)\in C^\infty(\BR^n)$ satisfying $|D^2\bar{u}|\leq K$ for some $K>0$ and $1+\ld_i\ld_j\geq0$.  According to the Bernstein theorem \cite{TW02}*{Theorem A}, $(x,D\bar{u}(x))$ must be an affine subspace.  Since $\sum_k(\pl_k\bar{u})^2 \leq c$, the affine subspace must be $\BR^n\times\{a\}$.
\end{proof}

%%%%%%%%%%%%%%%%%%%%%%%%%%%%%%%%
\subsection{Convergence to Self-Expanders}

A solution $F:L\times(0,\infty)\to\BR^n$ to the mean curvature flow is called a self-expander if the submanifold $F(L,t)$ is the dilation of $F(L,1)$ with the factor $\sqrt{t}$.  It is a natural model for immortal solutions to the mean curvature flow, and can also be used to study the mean curvature flow of conical singularities.  In our setting $F(x,t) = (x,Du)$, one finds that being a self-expander means that
\begin{align}
    u(x,1) &= t\cdot u\left( \frac{x}{\sqrt{t}},1 \right)
\end{align}
for any $t>0$.  Therefore, for the entire Lagrangian mean curvature flow \eqref{LMCF0}, $u:L\times(0,\infty)\to \BR^n$ is a self-expander if and only if $u_1 = u(\,\cdot\,,1)$ obeys
\begin{align}
    \frac{1}{\ii} \frac{\det(\bfI + \sqrt{-1}D^2u_1)}{\sqrt{\det(\bfI+(D^2u_1)^2)}} - u_1 + \oh\ip{Du_1}{x} &= 0
\label{self-exp0} \end{align}

In \cite{CCH12}*{Theorem 1.2}, it is shown that if the initial condition $u_0$ is strictly distance-decreasing, and $(x,Du_0(x))$ is asymptotic to a cone at spatially infinity, then the rescaled flow converges to a self-expander.  In the following proposition, we prove that the result holds true in the $2$-convex case.

\begin{thm} \label{thm_conv_expander}
    Let $u_0\in C^2(\BR^n)$ be a $2$-convex function with $\sup_{x\in\BR^n}|D^2 u_0|^2\leq c$ for some $c>0$, and
    \begin{align*}
        \lim_{\mu\to\infty} \frac{u_0(\mu x)}{\mu^2} &= U_0(x) ~,
    \end{align*}
    for some $U_0(x)$.  Let $u(x,t)$ be the solution to \eqref{LMCF0} given by Theorem \ref{thm_2convexini}.  Then, $u(\mu x,\mu^2 t)/\mu^2$ converges to a smooth self-expanding solution $U(x,t)$ to \eqref{LMCF0} in $C^\infty_{\text{loc}}(\BR^n\times(0,\infty))$ as $\mu\to\infty$.  As $t\to 0$, $U(x,t)$ converges to $U_0(x)$ in $C^0_{\text{loc}}(\BR^n)$.
\end{thm}

\begin{proof}
It is straightforward to see that for any $\mu>0$,
\begin{align*}
    u_\mu(x,t) &= \frac{1}{\mu^2}u(\mu x,\mu^2 t)
\end{align*}
is a solution to \eqref{LMCF0} with initial condition $u_{\mu}(x,0) = \mu^{-2}u_0(\mu x)$.  Since the spatial Hessian remains unchanged in this rescaling, it follows from Theorem \ref{thm_2convexini} that
\begin{align*}
    \left|D^\ell u_\ld(x,t)\right| &\leq c_\ell\, t^{1-\frac{\ell}{2}}
\end{align*}
for any $\ld>0$, $\ell>2$ and on $\BR^n\times(0,\infty)$.  With Remark \ref{rmk_dangle}, there exists for any $m\geq1$ and $\ell\geq0$ a constant $c_{m,\ell}>0$ such that
\begin{align*}
    \left| \frac{\pl^m}{\pl t^m}D^\ell u_\mu \right| &\leq c_{m,\ell}\,t^{1-m-\frac{\ell}{2}} ~.
\end{align*}

It is clear that $u_\mu(0,0)$ and $Du_\mu(0,0)$ are uniformly bounded for $\mu\geq1$.  These estimates together with the Arzel\`a--Ascoli theorem implies that $u_\mu(x,t)$ converges to some $U(x,t)$ in $C^\infty_{\text{loc}}(\BR^n\times(0,\infty))$ as $\mu\to\infty$.  Moreover, $U(x,t)$ satisfies the equation \eqref{LMCF0}, is $2$-convex for all $t$, and $*\Om\geq\vep_1$ for all $t$.  It follows from the construction that $U(x,t)$ satisfies $U(\mu x,\mu^2t) = \mu^2\,U(x,t)$ for any $\mu>0$, and thus $U(x,t)$ is a self-expander.

Since $|\pl U/\pl t|\leq c_{1,0}$, $U(x,t)$ converges to a function $U(x,0)$ as $t\to0$.  Again by $|\pl u_\mu/\pl t|\leq c_{1,0}$, one finds that
\begin{align*}
    U(x,0) &= \lim_{t\to 0}\lim_{\mu\to\infty} u_\mu(x,t) = \lim_{\mu\to\infty} u_\mu(x,0) = U_0(x) ~.
\end{align*}
It finishes the proof of this proposition.
\end{proof}

On the other hand, it is shown in \cite{CCH12}*{Lemma 7.1} that for any $U_0(x)$ which is strictly distance-decreasing and homogeneous of degree $2$, there exists a self-expander to \eqref{LMCF0} which is strictly distance-decreasing for all time.  A function is said to be homogeneous of degree $2$ if
\begin{align*}
    U_0(x) = \frac{1}{\ld^2}U_0(\ld x)
\end{align*}
for all $\ld>0$.  If $U_0$ is $C^2$ over $\BR^n$, we can invoke Theorem \ref{thm_longt_C2}.  It is more interesting to assume that $U_0$ is $C^2$ only on $\BR^n\setminus\{0\}$, and it requires some more work to construct the solution.

\begin{prop} \label{prop_exist_expander}
    Suppose that $U_0:\BR^n\to\BR$ is homogeneous of degree $2$, is $C^2$ on $\BR^n\setminus\{0\}$, and satisfy \eqref{bound1} on $\BR^n\setminus\{0\}$.   Then, \eqref{LMCF0} admits a unique self-expanding solution $u(x,t)$ in the space $C^0(\BR^n\times[0,\infty))\cap C^\infty(\BR^n\times(0,\infty))$ with initial condition $U_0$ such that $u(x,t) = t\cdot u(x/\sqrt{t},1)$ and \eqref{bound1} is preserved along the flow.

    Moreover, there exists $c_\ell = c_\ell(\vep_1)$ for any $\ell>2$ such that $\sup_{x\in\BR^n}|D^\ell u|^2 \leq c_\ell t^{2-\ell}$ for any $t>0$.
\end{prop}

It follows from the assumption that $DU_0$ is uniformly Lipschitz.  The following lemma is analogous to Lemma \ref{lem_booster}, and will be used to handle $U_0$ near the origin.

\begin{lem} \label{lem_booster0}
    Given any $\tau>1$,  there exists a sequence of smooth functions $\{E_k(x)\}_{k\in\BN}$ with the following significance.
    \begin{enumerate}
        \item $D^2E_k = \bfI$ when $|x|\leq 1/k$.
        \item $0< D^2E_k\leq \tau$ everywhere.
        \item At any $x\in\BR^n$, the ratio between any two eigenvalues of $D^2E_k(x)$ belongs to $[1/\tau,\tau]$.
        \item For any $R>0$, $E_k$ converges to $0$ uniformly over $\overline{B_R(0)}$ as $k\to\infty$, so do their derivatives.
    \end{enumerate}
\end{lem}

\begin{proof}
As in the proof of Lemma \ref{lem_booster}, let $E = E(r)$ where $r = |x|$.  Denote $E'(r)$ by $e(r)$.
Let
\begin{align}
    u(r) &= e'(r) \frac{r}{e(r)} ~,
\end{align}
and then
\begin{align}
    e(r) &= e(r_0) \exp\left(\int_{r_0}^r \frac{u(\rho)}{\rho}\dd\rho \right) ~.
\end{align}

For any $k\in\BN$, choose a smooth function $u_k(r)$ for $r\geq0$ with
\begin{itemize}
    \item $u_k(r) = 1$ if $r\leq 1/k$;
    \item $u_k(r) = 1/\tau$ if $r\geq 2/k$;
    \item $u'_k(r)\leq 0$ for all $r$.
\end{itemize}
Set $r_0$ to be $1/k$, and set $e_k(1/k)$ to be $1/k$.  It follows that $e_k(r) = r$ for $r\leq k$.  When $r\geq 1/k$,
\begin{align*}
    e_k(r) &= \frac{1}{k}\exp\left(\int_{1/k}^r \frac{u(\rho)}{\rho}\dd\rho\right)
    \leq \frac{1}{k}\exp\left(\int_{1/k}^r \frac{1}{\rho}\dd\rho\right) = r ~.
\end{align*}
When $r\geq 2/k$,
\begin{align*}
    e_k(r) &= e_k(2/k)\exp\left(\int_{k/2}^{r} \frac{u(\rho)}{\rho}\dd\rho\right) \leq \left(\frac{2}{k}\right)^{1-\frac{1}\tau} {r^{\frac{1}{\tau}}} ~.
\end{align*}
The function $E_k(r)$ is determined by setting $E_k(0) = 0$.  It is not hard to see that the corresponding $E_k(r)$ satisfies the assertions of this lemma.
\end{proof}

\begin{proof}[Proof of Proposition \ref{prop_exist_expander}]
To start, note that there exists a $\dt>0$ such that $\dt\leq D^2(U_0+E_k)\leq 1/\dt$ on where $0<|x|\leq 1/k$.  With this observation, the argument of Theorem \ref{thm_longt_C2} works for the initial condition $U_0+E_k$.  Denote the solution by $\td{u}^k$.  By the same argument as that in Theorem \ref{thm_longt_C2}, $\td{u}^k$ (subsequentially) converges on compact subsets of $\BR^n\times[0,\infty)$, and converges smoothly on compact subsets of $\BR^n\times(0,\infty)$.  The limit, $u$, is a solution to \eqref{LMCF0} with initial condition $U_0$.

Since $U_0$ is homogeneous of degree $2$, $u_\mu(x,t) = \mu^{-2}\,u(\mu x,\mu^2 t)$ (for any $\mu>0$) is also a solution to \eqref{LMCF0} with initial condition $U_0$.  It follows from the uniqueness theorem in \cite{CP09} that $u_\mu(x,t) = u(x,t)$ for any $\mu>0$.  Hence, the solution is a self-expander.
\end{proof}

With the same argument as that in Theorem \ref{thm_2convexini}, Proposition \ref{prop_exist_expander} can be pushed to the $2$-convex case.

\begin{prop}
    Let $U_0:\BR^n\to\BR$ be homogeneous of degree $2$.  Suppose that on $\BR^n\setminus\{0\}$, $U_0$ is $C^2$, satisfies $*\Om\geq\vep_1$ for some $\vep_1>0$, and is $2$-convex.   Then, \eqref{LMCF0} admits a unique solution $u(x,t)$ in the space $C^0(\BR^n\times[0,\infty))\cap C^\infty(\BR^n\times(0,\infty))$ with initial condition $U_0$ such that $u(x,t) = t\cdot u(x/\sqrt{t},1)$.  Moreover, the two conclusions of Theorem \ref{thm_2convexini} hold true for $u(x,t)$.
\end{prop}

%%%%%%%%%%%%%%%%%%%%%%%%%%%%%%%%%%%%%%%%%%%%%%%%%%%%%%%%%%%%%%%%

\end{document}